\documentclass{conm-p-l}
\usepackage{amssymb}
\usepackage{graphicx}

\newcommand{\ignore}[1]{}

\newcommand{\tr}{\operatorname{tr}}

\newcommand{\SL}{\operatorname{SL}}

\newcommand{\bb}{\mathbb} 

\newcommand{\C}{\bb C} 

\newcommand{\Z}{\bb Z}

\newcommand{\R}{\bb R}
\newcommand{\N}{\bb N}

\newcommand{\Q}{\mathbb Q}

\newcommand{\diag}{\operatorname{diag}}

\newcommand\norm[1]{\left\|#1\right\|}

\newcommand\abs[1]{\left|#1\right|}

\newcommand\set[1]{\left\{{#1}\right\}}

\def\rsemi{\hbox{$\triangleright\!\!\!<$}}

\newtheorem{Theorem}{Theorem}
\newtheorem{Cor}[Theorem]{Corollary}
\newtheorem{Prop}[Theorem]{Proposition}
\newtheorem{Lemma}[Theorem]{Lemma}
\newtheorem*{lemma*}{Lemma}

\newtheorem{remark}[Theorem]{Remark}
\newtheorem*{theorem*}{Theorem}
\newtheorem{Def}[Theorem]{Definition}
\numberwithin{equation}{section}
\numberwithin{Theorem}{section}

\begin{document}
\title[Exponents of Diophantine approximation]{ Diophantine approximation exponents \\ on homogeneous varieties}
\author{Anish Ghosh, Alexander Gorodnik, and Amos Nevo} 
\address{School of Mathematics, Tata Institute of Fundamental Research, Mumbai, India and School of Mathematics, University of East Anglia, Norwich, UK}
\email{ghosh.anish@gmail.com}
\address{School of Mathematics, University of Bristol, Bristol UK }
\email{a.gorodnik@bristol.ac.uk}
\address{Department of Mathematics, Technion IIT, Israel}
\email{anevo@tx.technion.ac.il}

\date{\today}
\thanks{The first author acknowledges support of  the Royal Society. The second author acknowledges
  support of EPSRC, ERC and RCUK. The third author acknowledges support of ISF}

\begin{abstract}
Recent years have seen very important developments at the interface of Diophantine approximation and
homogeneous dynamics. In the first part of the paper we give a brief exposition of a dictionary developed by Dani and
Kleinbock-Margulis which relates Diophantine properties of vectors to distribution of orbits of flows
on the space of unimodular lattices.  In the second part of the paper we  briefly describe an extension of this dictionary recently developed by the authors, which establishes  an analogous dynamical correspondence for general lattice orbits  on homogeneous spaces. 
We concentrate specifically on the problem of estimating exponents of Diophantine approximation by arithmetic lattices acting on algebraic varieties.
In the third part of the paper, we exemplify our results by  establishing explicit bounds for the Diophantine exponent  of dense lattice orbits in a number of basic cases.
 These include the linear and affine actions on affine spaces, and the action on the variety of matrices of fixed determinant. 
In some cases, these exponents are shown to be best possible. 
\end{abstract}

\maketitle
\begin{center}
\dedicatory{\emph{To S.G.~Dani on the occasion of his 65th birthday.}}
\end{center}
{\small
\tableofcontents
 }

\section{Introduction}

The theory of Diophantine approximation has many deep and fruitful connections with dynamical properties
of flows on homogeneous spaces. This connection has provided many fundamental new insights enriching both fields.
For instance, E.~Artin \cite{artin}
used continued fractions to construct dense geodesics on the modular surface, and
it was observed by C.~Series \cite{series}
that  classical continued fraction expansions can be constructed as cutting sequences for orbits of the geodesic
flow on the modular surface. 

This remarkable connection between Diophantine approximation and dynamics also exists in higher dimensions. 
It was realised by S.G.~Dani \cite{Dani} that Diophantine properties of vectors in Euclidean space
can be encoded by orbits of a suitable one-parameter flow on the space of unimodular lattices.
In particular, he showed that badly approximable vectors correspond to bounded orbits, and singular
vectors correspond to divergent orbits. This work has inspired many subsequent investigations exploring properties
of flows on homogeneous spaces, and the techniques developed gave rise to the solution of several longstanding open
problems in number theory. 
For example, one can mention such notable advances as the solution of Sprindzhuk's conjecture
in the theory of Diophantine approximation with dependent quantities by D.~Kleinbock and G.~Margulis \cite{KM1}
and the computation of the Hausdorff dimension of the set of singular vectors in $\mathbb{R}^2$ by
Y.~Cheung \cite{cheung}.

In this paper,  we first discuss the correspondence between Diophantine properties of vectors and recurrence
properties of flows developed by Dani, Kleinbock and Margulis. We then put these results in the context of recent works
of the authors on Diophantine exponents on homogeneous varieties of semisimple algebraic groups. Finally, we present some new estimates
for Diophantine exponents on specific homogeneous varieties, complementing the results of \cite{GGN3}.

\section{Diophantine approximation and the  shrinking target property}\label{sec:overview}

\subsection{Classical Diophantine approximation}
It is a well-known theorem of Dirichlet 
 that given a vector $x\in \mathbb{R}^d$, for every $R>1$ one can find
$m\in\Z^d$ and $n\in\N$ such that
\begin{equation}\label{eq:dir}
\left\|x-\frac{m}{n}\right\|_\infty\le n^{-1}R^{-1/d}\quad\hbox{and}\quad n\le R.
\end{equation}
\noindent Here $\|~\|_{\infty}$ denotes the maximum norm. Dirichlet introduced his famous pigeonhole principle to prove the above theorem. A proof can also be provided using Minkowski's convex body theorem. We refer the reader to \cite{Schmidt3} for details. 
In particular, it follows from (\ref{eq:dir}) that the inequality 
\begin{equation}\label{eq:dioph}
\left\|x-\frac{m}{n}\right\|_\infty\le n^{-1-1/d}
\end{equation}
always has a solution with $m\in\Z^d$ and $n\in\N$. 
We note that the theorem above is valid more generally for systems of linear forms. The properties introduced below and many theorems in this section are also valid in this setting. The Diophantine setting considered above is referred to as \emph{simultaneous} Diophantine approximation. 

Diophantine properties of vectors can also be studied in the context of linear forms, i.e. given $x \in \mathbb{R}^d$ one can study small values of the linear form 
$$ |m_1x_1 + \dots + m_d x_d + n| $$
for $m = (m_1, \dots, m_d) \in \mathbb{Z}^d$ and $n \in \mathbb{Z}$. This is referred to as the
\emph{linear} setting. These two settings are related by Khinchin's transference principle \cite{Schmidt3}.\\

One is interested in vectors in $\R^d$ for which the general estimates (\ref{eq:dir}) and
(\ref{eq:dioph})
can - or cannot - be improved; in
particular, one wishes to study the size of sets of such vectors. A vector $x$ is called {\it badly
  approximable} if there exists $c>0$ such that the inequality 
\begin{equation}\label{bad}
\left\|x-\frac{m}{n}\right\|_\infty\le c\, n^{-1-1/d}
\end{equation}
has no solutions $m\in\Z^d$ and $n\in\N$.

At the other extreme, the vector $x$ is called {\it singular}
if for every $c>0$ and $R\ge R(c)$, the system of inequalities
\begin{equation}\label{singular}
\left\|x-\frac{m}{n}\right\|_\infty\le c\, n^{-1}R^{-1/d}\quad\hbox{and}\quad n\le R
\end{equation}
has a solution $m\in\Z^d$ and $n\in\N$. In other words, Dirichlet's theorem can be infinitely improved
for such vectors.
One can show that a vector is badly approximable (resp. singular) if and only if it also has the analogous property in the sense of the  linear setting.

We note that a number $x\in \R$ is badly approximable if and only if its continued fraction expansion 
has bounded digits, and a number $x\in \R$ is singular if and only if it is rational. These properties are more difficult to characterise in higher dimensions, but it turns out they have a very
convenient interpretation based on dynamics of certain flows on the space of unimodular lattices. 
Denote by $\mathcal{L}_{d+1}$ the space of unimodular lattices in $\R^{d+1}$, which can be identified with 
the homogeneous space $\SL_{d+1}(\bb Z) \backslash \SL_{d+1}(\bb R)$. It has an invariant probability measure as well as a metric, inherited from a left invariant metric on $\SL_{d+1}(\bb R)$. The quotient is non-compact and its compact subsets are described by Mahler's compactness criterion. For $x\in\R^d$ we define the lattice
$$
\Lambda_x:=\{(n, m-nx):\, m\in\Z^d,n\in\Z\}\in \mathcal{L}_{d+1}.
$$
We consider the action on $\mathcal{L}_{d+1}$ by the one-parameter subgroups
$$g_t := \diag(e^{-t}, e^{t/d},\dots, e^{t/d}).
$$
The following results follow from 
the work of Dani:

\begin{Prop}[Dani \cite{Dani}]\label{badly}
With notation as above,
\begin{enumerate}
\item[(i)]
$x \in \bb R^d$ is badly approximable if and only if the semiorbit $\Lambda_x g_t$, $t>0$,  is bounded in
$\mathcal{L}_{d+1}$.
\item[(ii)]
$x \in \bb R^d$ is singular if and only if the semiorbit $\Lambda_x g_t$, $t>0$, is divergent, i.e. leaves every compact set in $\mathcal{L}_{d+1}$.
\end{enumerate}
\end{Prop}

The idea of the proof of Proposition \ref{badly} is based on the observation that if the lattice
$\Lambda_x g_t$ contains a small non-zero vector, this gives a solution of the relevant
Diophantine inequalities. Indeed, let 
\begin{equation}\label{eq:omega}
\Omega(\delta):=\{\Lambda\in\mathcal{L}_{d+1}:\, \exists~z\in \Lambda-\{0\}:\, \|z\|_\infty< \delta\}.
\end{equation}
Then if for some $t>0$, we have $\Lambda_x g_t \in \Omega(\delta)$ with $\delta\in (0,1)$, the system of inequalities 
\begin{equation}\label{eq:class}
\left\|x-\frac{m}{n}\right\|_\infty < \delta e^{-t/d}\quad \hbox{and}\quad n<\delta e^t
\end{equation}
has a solution $m\in\Z^d$ and $n\in\N$.
By Mahler's compactness criterion, the family of sets $\Omega(\delta)$
form a basis of neighborhoods $\Omega(\delta)$ of infinity of $\mathcal{L}_{d+1}$. 
Thus, Diophantine properties of the vector $x$
are determined by visits of the semiorbit $\Lambda_x g_t$ to neighborhoods of infinity. 

\subsection{Schmidt's game and bounded orbits}
 The points whose $g_t$ semiorbits are bounded (resp. divergent) form a set of measure zero. Nevertheless, the former are quite abundant. In \cite{Dani2, Dani} Dani used the above correspondence (i.e. Proposition \ref{badly} (i)), along with Schmidt's results on his game to show that bounded orbits for certain partially hyperbolic flows on homogeneous spaces have full Hausdorff dimension. Schmidt's game was introduced in \cite{Schmidt1} and is played in a complete metric space $X$.  Two players, A and B start with a subset $W\subseteq X$, and two parameters $0<\alpha, \beta<1$.  The game consists of choosing a sequence of nested closed balls. Player A begins by choosing a  ball $A_0$ and B continues by choosing $B_0$ and so on. 
 $$A_0 \supset B_0 \supset A_1 \supset B_1\supset \dots $$
 \noindent The radii of the balls are related by the parameters $\alpha$ and $\beta$:
 $$ r(B_i) = \alpha r(A_i)~\text{ and }~r(A_{i+1}) = \beta r(B_i).$$ 
 \noindent Player B wins this game if $\bigcap_n A_n$ intersects $W$. The set $W$ is called $(\alpha, \beta)$-winning if Player B can find a winning strategy , $\alpha$-winning if it is $(\alpha, \beta)$-winning for all $0 < \beta < 1$ and winning if it is $\alpha$-winning for some $\alpha > 0$.
 
Schmidt games have many nice properties. Most prominently, a winning subset of $\R^d$ is \emph{thick}, i.e. the intersection of a winning set with every open set in $\R^d$ has Hausdorff dimension $d$. Schmidt showed that badly approximable vectors form an $\alpha$-winning set for $0 < \alpha \leq 1/2$. Subsequent to Dani's introduction of Schmidt's game in homogeneous dynamics, a general conjecture on abundance of bounded orbits was formulated by Margulis \cite{Mar2} in his Kyoto ICM address, generalizing Dani's results. The conjecture was proved in stages by Kleinbock and Margulis \cite{KMbounded} and Kleinbock and Weiss \cite{KWeiss}. There has been intense activity in this subject, and Schmidt games and their variations have been used to prove a wide variety of results. Recently in \cite{DS}, Dani and H. Shah introduced a new topological variant of Schmidt's game. 

\subsection{Dani's correspondence and Khintchine's theorem}

Proposition \ref{badly} and Proposition \ref{p:klma} below are examples of what Kleinbock and Margulis have termed the \emph{Dani correspondence}. In
their paper \cite{KM2}, they further developed this correspondence 
to handle inequalities 
of the form 
\begin{equation}\label{def:KG}
\left\|x-\frac{m}{n}\right\|_\infty\le n^{-1}\psi(n),
\end{equation}
\noindent where $\psi$ is a general
nonincreasing function, and to obtain a new proof of 
Khintchine's theorem\footnote{In fact, their paper deals with systems of linear forms, i.e. the
  Khintchine-Groshev theorem.} using homogeneous dynamics. Recall that Khintchine's theorem states that
the set of $x \in \bb R^d$ for which inequality (\ref{def:KG}) 
has infinitely many solutions $m\in \Z^d$ and $n\in\N$
has zero (resp. full) measure depending on the
convergence (resp. divergence) of the sum
$$ \sum_{n = 1}^{\infty} \psi(n)^d.$$
As before, the original question in Diophantine approximation is transfered to 
a problem about visits of the semiorbit $\Lambda_x g_t$ to a family of shrinking neighbourhoods
of infinity. The rate at which these neighbourhoods shrink is determined by the decay rate of the function
$\psi$. Kleinbock and Margulis then use the exponential mixing property of the flow $g_t$ in conjunction
with a very general and quantitative form of the Borel--Cantelli lemma due to Sprindzhuk, to
establish a zero-one law and deduce the
Khintchine-Groshev theorem as a corollary. As another corollary they also obtain logarithm laws for geodesic
excursions to shrinking neighbourhoods of cusps of locally symmetric spaces, thereby generalising
Sullivan's logarithm law. Further, they established versions of zero-one laws for multi-parameter actions,
thereby confirming, in stronger form, a conjecture of Skriganov in the geometry of numbers.

\subsection{Diophantine approximation on manifolds}
One says that the vector $x\in \R^d$ is {\it very well
  approximable} if there exists $\epsilon>0$ such that the inequality
$$
\left\|x-\frac{m}{n}\right\|_\infty\le m^{-1-1/d-\epsilon}
$$
has infinitely many solutions $m\in\Z^d$ and $n\in\N$. The exponent in (\ref{eq:dioph}) is the best possible, and it follows from the Borel-Cantelli
lemma that the set of $x\in \mathbb{R}^d$ which are very well approximable has zero Lebesgue measure. The subject of metric Diophantine approximation on manifolds seeks to study the extent to which generic Diophantine properties are inherited by proper submanifolds of $\bb R^d$. In $1932$, Mahler's investigations into the classification of numbers according to their Diophantine properties, led him to conjecture that almost every point on the curve
$$ (x, x^2,\dots, x^d) $$
\noindent is not very well approximable. Mahler's conjecture was resolved by Sprindzhuk who in turn conjectured a more general form of his theorem. Let $M$ be a $k$-dimensional submanifold of $\R^d$ parametrised by a $C^l$-map $f:U\to M$.
We say that $M$ is nondegenerate if for almost every $x\in U$, the space spanned by the partial derivatives
of $f$ up to order $l$ coincides with $\R^d$. The following was a long standing conjecture of Sprindzhuk
(in the case of analytic manifolds) and proved by Kleinbock and Margulis\footnote{In fact, they proved more general \emph{multiplicative} versions of the conjecture.}:

\begin{Theorem}[Kleinbock--Margulis \cite{KM1}]\label{th:km0}
Almost every point on a nondegenerate submanifold $M\subset \R^d$ is not very well approximable.
\end{Theorem} 

The proof once more follows a dynamical route. Namely,  it was observed by Kleinbock and Margulis that the property of being very well approximable also has a convenient interpretation in terms of dynamics on the space of unimodular 
lattices:

\begin{Prop}[Kleinbock--Margulis \cite{KM1}]\label{p:klma}
A vector $x\in \R^d$ is very well approximable if and only if there exists $\alpha>0$ such that
$\Lambda_x g_{t_i} \in \Omega(e^{-\alpha t_i})$ for a sequence $t_i\to\infty$.
\end{Prop} 

Hence, in order to improve the exponent of Diophantine approximation for a vector $x\in\mathbb{R}^d$, 
one needs to establish
that the semiorbit $\Lambda_x g_t$ visits the sequence of exponentially shrinking sets
$\Omega(e^{-\alpha t})$ infinitely often. This is a common feature of many chaotic dynamical systems,
and is usually called the \emph{shrinking target property}. With the help of Proposition \ref{p:klma}, the proof of Theorem \ref{th:km0} reduces to an analysis
of visits of translated submanifolds $\Lambda_{f(U)}g_t$ to the neighbourhoods $\Omega(e^{-\alpha t})$.
The crucial and the most difficult part of the argument is
an explicit estimate on the measure of the set of $x\in U$ such that $\Lambda_{f(x)}g_t\in
\Omega(e^{-\alpha t})$. This estimate generalises the non-divergence properties of 
unipotent flows discovered by Margulis in \cite{Mar} and developed further in quantitative forms by Dani in \cite{Dani2,Dani3}. 

 The relevant result is stated as follows. Let $C, \alpha$ be positive numbers, $B$ an open subset of $\bb R^k$ and $\lambda$ denote Lebesgue measure of appropriate dimension.  Say that a function $f : B \to \bb R$ is $(C, \alpha)$-good on $B$ if for any open ball $J \subset B$ and any $\epsilon > 0$,
$$ \lambda(\{x \in J~:~|f(x)| < \epsilon\}) \leq \left(\frac{\epsilon}{\sup_{x \in J}|f(x)|} \right)^{\alpha} \lambda(B).$$

\noindent The main property of unipotent flows which allows for nondivergence is precisely the $(C,
\alpha)$-good property. Kleinbock and Margulis showed that more generally, smooth nondegenerate maps also
have this property. Let $\mathcal{L}(\bb Z^d)$ be the poset of primitive subgroups of $\bb Z^d$. 
For discrete subgroups $\Lambda$ we define $\|\Lambda\|$ as the norm of the corresponding vector 
in a suitable wedge product. Now we state the main
estimate in \cite{KM1}, which plays crucial role in the proof of Theorem \ref{th:km0}:

\begin{Theorem}
Let an open ball $B(x_0,r_0) \subset \bb R^k, C, \alpha > 0, 0 < \rho < 1/d$ and a continuous map $h :
B(x_0,3^dr_0) \to \SL_d(\bb R)$ be given. We assume that for every $\Delta \in \mathcal{L}(\bb Z^d)$,
\begin{enumerate}
\item the function $x \to \|h(x)\Delta\|$ is $(C, \alpha)$-good on $B(x_0,3^dr_0)$,
\item $\sup_{x \in B(x_0,r_0)}\|h(x)\Delta\| \geq \rho$.
\end{enumerate}
Then for every $\epsilon\in (0,\rho]$,
$$ \lambda(\{x \in J~:~\bb Z^d h(x) \notin\Omega(\epsilon)\}) \leq D(d,k) C \left(\frac{\epsilon}{\rho}\right)^{\alpha}\lambda(B), $$
where $D(d,k)$ is a constant depending on $d$ and $k$ only.
\end{Theorem}

We refer the reader to the survey \cite{Kleinbock-survey} for further developments and applications of the above theorem.

 We note that in the theory of metric Diophantine approximation on manifolds, one is concerned with
 approximating points on manifolds by \emph{all rational points in the ambient Euclidean space}. We now turn in the next section to discuss a completely different, but equally natural, question.  Namely we will consider Diophantine 
approximation of a general point on a variety intrinsically, by \emph{rational points lying on the variety itself}. 

\subsection{Intrinsic Diophantine approximation on algebraic varieties}
The question of Diophantine approximation on algebraic varieties by rational points on the variety itself was raised already half a century ago by S. Lang \cite{Lang},
but the results in this direction are still very scarce. 
In some cases one can deduce that 
the set $X(\Q)$ of rational points on an algebraic variety $X$ is dense in the set $X(\R)$ of real
points using a rational parametrisation of $X$ (for example, the stereographic projection for the quadratic
surfaces), but this approach usually provides poor bounds on Diophantine exponents that depend on the
degree of the parametrisation map. More generally, one can consider the problem of Diophantine approximation
by the set $X(\Z[1/p])$ of $\Z[1/p]$-points in $X$. Here even establishing density is a nontrivial task.

One of the most natural examples of algebraic varieties 
with rich structure of rational points is given by algebraic groups and their homogeneous spaces.
Here several results regarding quantitative density of rational points have been proved.
This includes elliptic curves and abelian varieties \cite{W}, general 
homogeneous spaces of semisimple algebraic groups \cite{GGN1,GGN2}, and
quadratic surfaces \cite{Drutu,schmutz,K-,FKMS}. In the latter two cases,
one can also use dynamical correspondences which 
relates Diophantine properties of points to shrinking target properties of orbits for the corresponding
dynamical systems. Let us now turn to a brief description of these correspondences.

\vspace{0.2cm}

Let $X$ be an algebraic variety in $\C^d$ defined over $\Q$ equipped with an action of 
a connected almost simple algebraic group $G\subset \hbox{GL}_d(\C)$ defined over $\Q$.
For simplicity of exposition, let us consider here the problem of Diophantine approximation in $X(\R)$ by rational points in
$X(\Z[1/p])$ where $p$ is prime.
A basic observation is that since rational points on $X$ can be parametrized using orbits of the group $G(\Z[1/p])$, they can be studied using
techniques from the theory of dynamical systems. The relevant dynamical system here is the space 
$$
Y=G(\Z[1/p]) \backslash (G(\R)\times G(\Q_p))
$$
with the action of the group $G(\Q_p)$ by right multiplication. 
The following proposition is an analogue the classical Dani correspondence described in the previous section:

\begin{Prop}[Ghosh--Gorodnik--Nevo \cite{GGN1}]\label{p:ggn}
Given $x\in X(\R)$, there exists $y_x\in Y$ and a sequence of neighborhoods $\mathcal{O}_\epsilon$
of the identity coset in $Y$ such that if 
$$
 y_x \cdot b \in \mathcal{O}_\epsilon\quad \hbox{for some $b\in G(\Q_p)$ with $\|b\|_p\le R$,}
$$
then the system of inequalities 
$$
\left\|x-\frac{m}{n}\right\|\le \epsilon\quad\hbox{and}\quad n\le c(x)\, R
$$
has a solution $\frac{m}{n}\in X(\Z[1/p])$, with the constant $c(x)$ uniform over $x$ in compact sets.
\end{Prop}

Proposition \ref{p:ggn} shows that the problem of Diophantine approximation reduces to 
the shrinking target problem for the orbit $ y_x G(\Q_p)$ with respect to  the sequence of neighbourhoods $\mathcal{O}_\epsilon$.

Let us assume that the set $X(\Z[1/p])$ is not empty and the group $G$ is isotropic over $\Q_p$. 
Then the closure $\overline{X(\Z[1/p])}$ in $X(\R)$ is open and closed in $X(\R)$ (see \cite{GGN1}).
In particular, when $X(\R)$ is connected, $X(\Z[1/p])$ is dense.
To measure the quality of Diophantine approximation, we introduce the notion of Diophantine approximation
exponents.

\begin{Def}\label{def:DE0}
{\rm  
Assume that  for $x\in \overline{X(\Z[1/p])}$, there exist constants $c=c(x)$ and
$\epsilon_0=\epsilon_0(x)$, such that for all $\epsilon < \epsilon_0$, the system of inequalities 
\begin{equation}\label{eq:d}
\left\|x-\frac{m}{n}\right\|\le \epsilon\quad\hbox{and}\quad n\le c\, \epsilon^{-\kappa}
\end{equation}
has a solution $\frac{m}{n}\in X(\Z[1/p])$.  
Define the {\it Diophantine
approximation  exponent} $\kappa_p(x)$ as the infimum of $\kappa > 0$ such that the foregoing inequalities have a
solution.
}
\end{Def} 

A lower bound on the exponents $\kappa_p(x)$ can be deduced from the following pigeon-hole argument.
Let us introduce the growth exponent of the number of rational points
$$
a_p(X)=\sup_{\hbox{\tiny compact }K\subset X(\R)} \limsup_{R\to\infty} \frac{\log N_R(K,X(\Z[1/p])}{\log R}
$$
where $N_R(K,X(\Z[1/p])$ denotes the number of points $\frac{m}{n}\in K\cap X(\Z[1/p])$ such that $n\le R$.
One can show that if the constant $c=c(x)$ is uniform over compact sets in $X(\R)$, then 
\begin{equation}\label{eq:pigeon}
\kappa_p(X)\ge \frac{\dim(X)}{a_p(X)}.
\end{equation}

The question about upper bounds for $\kappa_p(X)$ is much deeper. Indeed, any such upper bound would
quantify density of $X(\Z[1/p])$ in $\overline{X(\Z[1/p])}$. In view of Proposition \ref{p:ggn}, 
this question can be answered by studying the problem of establishing quantitative equidistribution 
for orbits of $G(\Q_p)$ in $Y$. Let us therefore introduce the family of averaging operators given by 
$$
A_R:L^2(Y)\to L^2(Y): \phi\mapsto \frac{1}{|B_R|}\int_{B_R} \phi(yb)\, db
$$
where $B_R=\{b\in G(\Q_p):\, \|b\|_p\le R\}$. 
It was shown in \cite{GGN1} that there exists $C,\theta>0$
such that for all sufficiently large $R$,
\begin{equation}\label{eq:mean}
\|A_R(\phi)-P(\phi)\|\le C\, |B_R|^{-\theta} \|\phi\|_2,
\end{equation}
where $P$ is an explicit projection operator on $L^2(Y)$.
Let $\theta_p$ denote the supremum over $\theta$'s for which the estimate (\ref{eq:mean}) holds.
This parameter provides the crucial input to deduce the following upper bound on the Diophantine exponent:

\begin{Theorem}[Ghosh, Gorodnik, Nevo \cite{GGN1}]\label{th:ggn1}
With notation as above,
\begin{enumerate}
\item[(i)] For almost every $x\in \overline{X(\Z[1/p])}$,
$$
\kappa_p(x)\le (2\theta_p)^{-1} \frac{\dim(X)}{a_p(G)}.
$$
\item[(ii)] 
For every $x\in \overline{X(\Z[1/p])}$,
$$
\kappa_p(x)\le \theta_p^{-1} \frac{\dim(X)}{a_p(G)}.
$$
\end{enumerate}
Moreover, the constant $c=c(x)$ in (\ref{eq:d}) is uniform over $x$ in compact sets. 
\end{Theorem}

Typically, $a_p(G)\ge a_p(X)$, and in that case, it follows that $\theta_p=1/2$, the bound in Theorem \ref{th:ggn1}(i)
matches with the lower bound (\ref{eq:pigeon}), so that it is best possible.  
For instance, Theorem \ref{th:ggn1} gives a sharp bound for Diophantine
approximation by $\Z[1/p]$-points on the two-dimensional sphere.
We note that the exponent $\theta_p$ is closely related
to the integrability exponents of automorphic representations, which has been extensively studied
in relation to the generalised Ramanujan conjectures,
and explicit estimates on $\theta_p$ are available in a number of cases (see \cite{Sarnak,BloBru} for a
detailed account). 

\vspace{0.2cm}

A fruitful approach to the problem of Diophantine approximation by the set of all rational points on nonsingular quadratic surfaces 
has been developed in \cite{Drutu, K-,FKMS}. In particular, it turns out that the problem of
Diophantine approximation on the sphere is related to a shrinking target problem for a suitable 
one-parameter flow $g_t$ on the space $Y=\Gamma \backslash G$ where $G=\hbox{SO}(d,1)$ and $\Gamma$ is a subgroup 
of integral matrices in $G$. We note that $Y$ can be naturally embedded in the space of unimodular
lattices $\mathcal{L}_{d+1}$ and one can set $\Omega_Y(\delta)=Y\cap \Omega(\delta)$ where $\Omega(\delta)$ is
defined in (\ref{eq:omega}). Given a vector $x$ on the $d$ dimensional unit sphere $S^d$, one can associate a point $y_x\in Y$
such that the following dynamical correspondence holds:

\begin{Prop}[Kleinbock--Merrill \cite{K-}]\label{p:km}
Suppose that there exists $t>0$ such that $y_x g_t \in \Omega_Y(\delta)$. Then the system of inequalities
$$
\left\| x-\frac{m}{n}\right\|_\infty < \frac{2\delta^{1/2}e^{-t/2}}{n^{1/2}}\quad \hbox{and}\quad
n<\delta e^t
$$
has a solution with $\frac{m}{n}\in S^d$.
\end{Prop}

For instance, Proposition \ref{p:km} is used to prove  the following analogue of the classical Dirichlet
theorem (cf. (\ref{eq:dir}) in the case of the unit spheres:

\begin{Theorem}[Kleinbock--Merrill \cite{K-}] \label{th:km}
There exists $c>0$ such that for every $x\in S^d$ and $R>1$, the system of inequalities
$$
\left\| x-\frac{m}{n}\right\|_\infty \le \frac{c}{n^{1/2}R^{1/2}}\quad \hbox{and}\quad n\le R
$$
has a solution with $\frac{m}{n}\in S^d$.
\end{Theorem}

We also mention that this approach also allows to prove an analogue of the Khinchin--Groshev theorem \cite{K-} 
and to study the set of badly
approximable vectors \cite{K-} and well approximable vectors \cite{Drutu}
 in the context of the intrinsic Diophantine approximation.

\section{Diophantine exponents for group actions}
The problem of establishing Diophantine exponents discussed in the previous sections is an instance of a much more general problem, namely   establishing a rate of distribution for dense lattice orbits on homogeneous spaces of the ambient group. In \cite{GGN3} we developed a general approach to establishing quantitative density 
of orbits which is based on a \emph{duality principle} combined with a \emph{quantitative mean ergodic theorem}.
 Then we establish a dynamical correspondence similar
to the correspondences discussed in the previous sections. 
More explicitly, the duality principle implies that quantitative density of lattice orbits
$\Gamma x$ in the homogeneous space $Gx$ is equivalent to the quantitative density of the corresponding orbits of $H_x$
in the space $Y=\Gamma \backslash G$, where $H_x$ denotes the stabilizer of $x$ in $G$.

To motivate the discussion by a concrete example, consider a case where the approach mentioned above works especially well, namely when $\Gamma$ is a lattice in an algebraic subgroup $G$
of the group of affine transformations.
In classical inhomogeneous Diophantine approximation, one is interested in minimizing
the quantity $|nx-m+x_0|$ for given $x,x_0\in \R^d$, as $n$ varies over $\Z$ and $m$ varies over $\Z^d$.
This problem can be interpreted as establishing quantitative density for orbits of the semigroup 
$\Gamma= \Z^\times \ltimes \Z^d$ acting on $\R^d$. 
There has been considerable recent interest in the study of density of orbits for more general groups of
affine transformations. One example is $\Gamma=\SL_{2}(\bb Z)$ acting on the punctured plane, where the 
$\Gamma$-orbit of a point with irrational slope is dense. In \cite{LN1, LN2}, Laurent and
Nogueira have studied effective versions of this density. 
In \cite{MW}, Maucourant and Weiss have used
effective equidistribution results for horocycle flows to obtain effective results for dense
$\Gamma$-orbits on the plane, where $\Gamma$ is an arbitrary lattice in $\SL_{2}(\bb R)$. 

In our discussion below $\Gamma$ will denote for instance a discrete subgroup of the group $\hbox{Aff}(\R^d)$ of affine transformations of $\R^d$.
We equip $\hbox{Aff}(\R^d)$ with a norm which is a natural way to measure complexity of elements of
$\Gamma$. Studying effective density of $\Gamma$-orbits amounts to estimating Diophantine exponents 
which we define as follows. 

\begin{Def}\label{def:DE} 
{\rm
Assume that  for $x,x_0\in \overline{\Gamma x}$, there exist  constants $c=c(x,x_0)$ and
$\epsilon_0=\epsilon_0(x,x_0)$ such that for all $\epsilon < \epsilon_0$, 
the system of inequalities 
$$
\|\gamma^{-1}x- x_0\|_\infty\le \epsilon\quad\text{and}\quad\|\gamma\|\le  c\,\epsilon^{-\kappa}.
$$ 
has a solution $\gamma\in \Gamma$.
Define the Diophantine
approximation  exponent $\kappa_\Gamma(x,x_0)$ as the infimum of $\kappa > 0$ such that the foregoing inequalities have a
solution.
}
\end{Def} 

 We note that the exponent defined above generalizes the Diophantine exponent for \emph{uniform} approximation by $\SL_{2}(\bb Z)$-orbits in $\bb R^2$ as considered by Laurent and Nogueira \cite{LN1}.
The exponent $\kappa_{\Gamma}$ above is naturally related to the exponent $\kappa_p$ discussed in the
previous section (which is associated with dense orbits of the group $\Gamma=G(\Z[\frac1p])$).

A basic geometric argument leads to a lower bound on the Diophantine exponent.
Let us define the growth exponent of $\Gamma$-orbits by
$$
a_\Gamma(x)=\sup_{\hbox{\tiny compact }K\subset X(\R)} \limsup_{R\to\infty} \frac{\log N_R(K,x)}{\log R}
$$
where $N_R(K,x)$ denotes the number of elements $\gamma\in \Gamma$ such that $\|\gamma\|\le R$ and
$\gamma^{-1}x$ belongs to $K$. It is not hard to show that for almost every $x_0\in X$,
$$
\kappa_\Gamma(x,x_0)\ge \frac{\dim(X)}{a_\Gamma(x)}.
$$
The quantity $a_\Gamma(x)$ could be difficult to estimate in general, but if the variety $X$ is homogeneous,
it can be estimated in terms of volume growth of a suitable subgroup.  $X$ can then be identified with the homogeneous space $G/H$
where $H$ is a closed subgroup of $G$. We set 
$$
a(X)=\limsup_{t\to\infty} \frac{\log m_H(H_t)}{t},
$$
where $H_t=\{h\in H;\, \log \|h\|\le t\}$ and $m_H$ is a right-invariant Haar
measure on $H$. Then one can show using discreteness of $\Gamma$ that for every $x\in X$,
$$
a_\Gamma(x)\le a(X).
$$
In particular, for almost every $x_0\in X$,
\begin{equation}\label{eq:llow}
\kappa_\Gamma(x,x_0)\ge \frac{\dim(X)}{a(X)}.
\end{equation}
Thus the fundamental question that arises is to determine when this lower bound is in fact sharp (for almost all $x$), and in general to give an estimate for the upper bound.

We show that in the above homogeneous setting
one can reduce the original problem of quantitative
density of $\Gamma$-orbits in $X$ to the problem of quantitative density of the corresponding $H$
orbits in the space $Y:=\Gamma \backslash G$. More precisely, we have the following dynamical correspondence:

\begin{Prop}[Ghosh--Gorodnik--Nevo \cite{GGN3}]
Let $x=gH\in X$, $x_0=g_0H\in X$  and $y=\Gamma g \in Y$, $y_0=\Gamma g_0 \in Y$. 
There exists a sequence of neighbourhoods $\mathcal{O}_\epsilon(y_0)$ of $y_0$
such that if there exists $h\in H$ such that 
$$
y\cdot h\in \mathcal{O}_\epsilon(y_0)\quad \hbox{and}\quad \|h\|\le R,
$$
then there exists $\gamma\in \Gamma$ such that
$$
\|\gamma^{-1}x-x_0\|_\infty\le \epsilon\quad \hbox{and}\quad \|\gamma\|\le c(x,x_0) R,
$$
where $c(x,x_0)$ is uniform over $x,x_0$ in compact sets. 
\end{Prop}

Hence, the problem of establishing upper bounds on $\kappa_\Gamma(x,x_0)$
is closely related to the shrinking target problem for the orbit $yH$ in $Y$
with respect to the family of neighbourhoods $\mathcal{O}_\epsilon(y_0)$. 
The later can be approached using a quantitative mean ergodic theorem 
for the action of $H$ on $Y$. 

Now we assume that the discrete group $\Gamma$ has finite covolume in $G$ and
consider a family of averaging operators
$$
\pi_Y(\beta_t):L^2(Y)\to L^2(Y): \phi\mapsto \frac{1}{m_H(H_t)}\int_{H_t} \phi(y h)\, dm_H(h).
$$
Let us suppose that there exist $C,\theta>0$
such that for all sufficiently large $t$,
\begin{equation}\label{eq:mean1}
\left\|\pi_Y(\beta_t)(\phi)-\int_Y \phi\, dm_Y \right\|\le C\, m_H(H_t)^{-\theta} \|\phi\|_2,
\end{equation}
where $m_Y$ denotes the normalised Haar measure on $Y$.
Let $\theta_\Gamma(X)$ denote the supremum over $\theta$'s for which the estimate (\ref{eq:mean}) holds.

\begin{Theorem}[Ghosh--Gorodnik--Nevo \cite{GGN3}]\label{th:ggn3}
For every $x_0\in X$ and almost every $x\in X$,  
$$
\kappa_\Gamma(x,x_0)\le (2\theta_\Gamma(X))^{-1}\frac{\dim(X)}{a(X)}.
$$
Moreover, the constants $c(x,x_0)$ and $\epsilon_0(x,x_0)$ in Definition \ref{def:DE}
are uniform over $x,x_0$ in compact sets. 
\end{Theorem}

Looking at the case $\theta_\Gamma(X)=1/2$ we have the following sample conclusion:

\begin{Cor}\label{cor:ggn3}
If $G$ and the stability group $H$ are semisimple and non-compact, and the representation of  $H$ 
on $L^2_0(\Gamma\backslash G)$ is a tempered representation of $H$, then the Diophantine exponent of $\Gamma$-action on $X=G/H$ is best possible, and is given by 
$$
\kappa_\Gamma(x,x_0) = \frac{\dim(X)}{a(X)}
$$
for every $x_0$ and almost every $x$ in $ X$.
\end{Cor}

We note that the results established in Theorem \ref{th:ggn3} and Corollary \ref{cor:ggn3}
hold not only for linear and affine action on varieties in $\R^d$, but also for linear and affine actions over other local
fields and more generally for actions on general homogeneous spaces of locally compact groups
satisfying some natural assumptions. We refer to \cite{GGN3} for general statements of these
results. In this general setting one defines the dimension of $X$ as
$$
\dim(X)=\limsup_{\epsilon\to 0}\frac{\log m_{X}(B_\epsilon(x_0))}{\log \epsilon},
$$
where $B_\epsilon(x_0)$ is the $\epsilon$-ball around $x_0$ in $X$, and $m_X$ is the measure on $X$
induced by the chosen measures on $G$ and $H$. We note that under the assumption imposed in \cite{GGN3}
this limit is independent of $x_0$.

The bound obtained in Theorem \ref{th:ggn3} for the approximation exponent depends  
 on spectral information pertaining to automorphic representations 
$L^2(\Gamma\backslash G)$ as well as arithmetic data related to the group and the lattice. As such, it is a highly
non-trivial task to compute these parameters in any given example. One of the main advantages in our approach is that this task is feasible in many interesting cases, and in fact often leads to best possible results. This necessarily involves a detailed study  of the
spectral theory of unitary representations of semisimple groups, and we now turn to demonstrating the method in some cases.

\section{Examples of Diophantine exponents}\label{sec:examples}

In this section we give some concrete examples of explicit estimates of Diophantine exponents arising in Theorem
\ref{th:ggn3}, complementing some of the examples presented in \cite{GGN3}.
For simplicity of exposition, let  now $F$ denote the fields $\R$, $\C$ or $\Q_p$.\footnote{Although we
  do not treat this case, the methods in \cite{GGN3} are very general and also hold for local fields of
  positive characteristic. In this setting one also has the advantage of better spectral estimates in
  certain instances, for example arising from work of Drinfeld and Lafforgue.}, Let $G$ be a linear algebraic subgroup
of the group $\hbox{SL}_n(F)\rsemi F^n$ considered as a group of affine transformations of $F^n$.
We fix a norm on $\R^n$ and $\C^n$, and a (vector space) norm on $\hbox{M}_n(\R)$ and $\hbox{M}_n(\C)$. In the local field case we take the standard  valuation on the field, and the standard maximum norm on the linear space $F^n$, and on $\hbox{M}_n(F)$.  
We view the affine group $\hbox{SL}_{n}(F)\rsemi F^{n}$, $ n \ge 2$  as a subgroup of $\hbox{SL}_{n+1}(F)$, specifically as  the stability group of the standard basis vector $e_{n+1}$,  and consider norms on it by restriction from $\hbox{SL}_{n+1}(F)\subset \hbox{M}_{n+1}(F)$. 
Let $X\subset F^n$ be an affine subvariety which is invariant and homogeneous under the $G$-action,
so that $X\simeq G/H$ where $H$ is closed subgroup of $G$.
We define the distance on $X$ by restricting the norm defined on $F^n$. Let $\Gamma$ be a lattice
subgroup of $G$ such that almost every  $\Gamma$-orbit is dense in $X$.   

We now proceed to describe several examples of classical Diophantine approximation problems in this
context and estimate the exponents $\kappa_\Gamma(x,x_0)$ defined in Definition \ref{def:DE}.

We remark that any choice of norm  on $\R^n$ or $\C^n$ (and hence on $X$) and on $\hbox{M}_n(\cdot)$ (and hence on $G$
and $H$) does not change the estimate of the exponent. We will thus  choose the norm most convenient for us, namely one whose restriction to $H$ has convenient properties.

For simplicity, we use notation $A\ll B$ if $A\le c\, B$ for some constant $c$,
and $A\asymp B$ if $c_1\, B\le A\le c_2\, B$ for some constants $c_1,c_2$.

\subsection{Homogeneous Diophantine approximation in linear space} 
Consider first  the classical case of homogeneous Diophantine approximation  in the linear action of a lattice subgroup of $\SL_n(F)$ on $F^n\setminus \set{0}$. 

\begin{Prop}\label{ex:4}
For $n \ge 3$  the exponent of Diophantine approximation for an arbitrary lattice
$\Gamma$ in $G=\SL_n(F)$ acting on $X=F^n\setminus \set{0}$ is estimated by 
$$
\frac{n}{(n-1)^2}\le \kappa_\Gamma(x,x_0) \le \frac{n}{n-1}
$$
for almost every $x,x_0\in F^n\setminus \set{0}$. 
\end{Prop} 

  It is also possible to give an upper bound for the Diophantine exponent for lattices in $\SL_2(F)$
  acting on the plane $F^2\setminus \set{0}$.   The bound depends on the lattice subgroup in question,
  and involves different and more elaborate considerations. In the case of $\SL_2(\Z)$ the bound obtained
  is $\kappa_\Gamma(x,x_0)\le 6$, and the same bound holds for any tempered lattice, namely any lattice $\Gamma$ for which $L^2_0(\Gamma \setminus G)$ is a tempered representation of $G$. For further details and results about Diophantine approximation by lattice orbits  in the real and complex plane we refer to \cite{GGN3}. 
 
\begin{remark}{\rm 
 \begin{itemize}
\item The only results we are aware of in the literature regarding estimates of the Diophantine exponent for lattice
  actions on homogeneous varieties are due to  Laurent--Nogueira \cite{LN1,LN2} and to Maucourant--Weiss
  \cite{MW}.  Laurent and Nogueira  established that generically $\kappa_\Gamma(x,x_0) \le
  3$ for $\SL_2(\Z)$ acting on the plane by explicitly constructing a sequence of approximants using 
a suitable continued fractions algorithm. Maucourant and Weiss have also established an (explicit, but not as sharp) upper bound for $\kappa_\Gamma(x,x_0)$ for arbitrary lattice subgroups of $\SL_2(\R)$ using effective equidistribution of horocycle flows.
\item We note that for the linear action, determining  the exact value of $\kappa_\Gamma(x,x_0)$ generically remains an open problem, for any lattice subgroup, over any field,  in any dimension. 
 \end{itemize}
} 
\end{remark}

\begin{proof}

{\it Part I : volume estimates.}
$F^n\setminus \set{0}$ is a homogeneous space of $\SL_n(F)$, and the stability $H$ group of the standard
basis vector $e_n$ is isomorphic to the group $\SL_{n-1}(F)\rsemi  F^{n-1}$. We choose a vector space norm on $\hbox{M}_n(F)$ with the property that its restriction to  $h\in H=\SL_{n-1}(F)\rsemi F^{n-1}$, written as $h=(h_1, v_1)$, is given by  
$$
\norm{(h_1, v_1)}= \max\set{\|h_1\|, \norm{v_1}},
$$
where $\|\cdot\|$ denotes the Euclidean norm when $F=\R,\C$ and the maximal norm when $F=\Q_p$.

We can now evaluate the volume growth of $H_t=\set{h\in H\,;\, \log \norm{h} < t}$.
Due to our choice of the norm,
$$
m_H(H_t)= \left(\int_{\|h_1\| \le e^{t}} dh_1\right)\cdot \left(\int_{\norm{v_1}\le e^{t} } dv_1\right)\,.
$$
The first integral can be estimated using \cite[Appendix~1]{DRS} for $F=\R,\C$ or 
\cite[Sec.~7]{GW} for general local fields. This gives
\begin{equation}\label{eq:sl}
\int_{\|h_1\| \le e^{t}} dh_1 \asymp
\left\{
\begin{tabular}{ll}
$e^{t\left((n-1)^2-(n-1)\right)}$ & \hbox{for $F=\R,\Q_p$,}\\
$e^{2t\left((n-1)^2-(n-1)\right)}$ & \hbox{for $F=\C$.}
\end{tabular}
\right.
\end{equation}
Also clearly,
$$
\int_{\|v_1\| \le e^{t}} dv_1 \asymp 
\left\{
\begin{tabular}{ll}
$e^{t (n-1)}$ & \hbox{for $F=\R,\Q_p$,}\\
$e^{2t (n-1)}$ & \hbox{for $F=\C$.}
\end{tabular}
\right.
$$
Hence,
$$
a(X)=\limsup_{t\to\infty} \frac{\log m_H(H_t)}{t} \asymp 
\left\{
\begin{tabular}{ll}
$(n-1)^2$ & \hbox{for $F=\R,\Q_p$,}\\
$2(n-1)^2$ & \hbox{for $F=\C$.}
\end{tabular}
\right.
$$
Now it follows from (\ref{eq:llow}) that for almost every $x_0\in X$,
$$
\kappa_\Gamma(x,x_0)\ge \frac{\dim(X)}{a(X)}=\frac{n}{(n-1)^2},
$$
which proves the lower bound in the proposition.\\

{\it Part II : spectral estimates.} 
Let $\pi$ denote the representation of $H$ on $L^2_0(\Gamma \setminus G)$.
We proceed to estimate the decay of the operator norm $\pi(\beta_t)$, where
$\beta_t$ are the Haar-uniform averages supported on the subsets $H_t$. 
The cases $n\ge 4$ and $n=3$ require separate arguments.

When $n\ge 4$, we use that for the group $\SL_k(F)$, any unitary representation 
without invariant vectors is $L^p$-integrable for $p>2(k-1)$, provided that $k \ge 3$ (see e.g. \cite{ht,oh}).
Hence, using the spectral transfer principle \cite{N3},  when $n\ge 4$ any unitary  representation $\sigma$ of $\SL_{n-1}(F)$ without invariant vectors has the property that $\sigma^{\otimes ( n-2) }$  is weakly contained in the regular representation  of $\SL_{n-1}(F)$. In particular, this statement holds for the representation of $\SL_{n-1}(F)$ of the form $\sigma=\pi|_{\SL_{n-1}(F)}$, for  any lattice subgroup  $\Gamma \subset \SL_n(F)$.
 Thus the restriction of the tensor power $ \pi^{\otimes (n-2)}$ to the closed subgroup $\SL_{n-1}(F)$ is
 weakly contained in the regular representation of  $\SL_{n-1}(F)$. By \cite{chh}, it follows that the
 $K$-finite matrix coefficients of the representation $\pi|_{\SL_{n-1}(F)}$ are dominated along the split
 Cartan subgroup by a scalar multiple of $\Xi_{n-1}^{1/(n-2)}$, where $\Xi_{n-1}$ denotes the Harish-Chandra function of $\SL_{n-1}(F)$. 
 
We claim that the last estimate also holds when $n=3$. To prove this, we consider the restriction of the representation $\pi$ to
$H = \SL_2(F)\rsemi F^2$. Since this representation has no $F^2$-invariant vectors, 
it follows from the Kazhdan's original argument \cite{k} that $\pi|_{\SL_{2}(F)}$ is weakly contained 
in the regular representation of $\SL_2(F)$, so that the same estimate as above applies.

Applying the general method of spectral estimates on groups with an Iwasawa decomposition (see
\cite{GGN3} for details), and using the fact that the sets $H_t$ are bi-invariant by a maximal compact
subgroup of $\SL_{n-1}(F)$, we can estimate the operator norm  by integrating the bound
$\Xi_{n-1}^{1/(n-2)}$ over  $H_t$, where we view the function as independent of the second variable
$v_1\in F^{n-1}$. This gives
\begin{align*}
\norm{\pi(\beta_t)} &\le \frac{1}{m_H(H_t)} \int_{H_t}\Xi_{n-1}^{1/(n-2)}(h_1) dh_1 dv_1\\
&=\frac{1}{m_H(H_t)}m_{F^{n-1}}(B_{e^{t}})  \int_{\|h_1\| \le
 e^{t}}\Xi_{n-1}^{1/(n-2)}(h_1)\, dh_1   
\end{align*}
Since the function $\Xi_{n-1}^{1/(n-2)} $ is $L^{2(n-2)+\eta}$-integrable for every $\eta>0$,
we can apply the following simple estimate, based on H\"older's inequality for the conjugate exponents
$\frac1p=\frac{1}{2(n-2)+\eta}$ and $\frac1q=1-\frac{1}{2(n-2)+\eta}$: 
\begin{align*}
\int_{\|h_1\| \le
 e^{t}}\Xi_{n-1}^{1/(n-2)}(h_1)\, dh_1 &\le \|\Xi_{n-1}^{1/(n-2)}\|_p \cdot \|\chi_{\{\|h_1\|\le e^t\}}\|_q\\
&\ll m_{\SL_{n-1}(F)}(\set{\|h_1\| \le e^{t}})^{1-\frac{1}{2(n-2)+\eta}}.
\end{align*}
Hence, using the above volume estimates, we deduce that 
\begin{align*}
\norm{\pi(\beta_t)} &\ll m_{\SL_{n-1}(F)}(\set{\|h_1\| \le e^{t}})^{-\frac{1}{2(n-2)+\eta}}
\ll  m_H(H_t)^{-\theta+\eta'}
\end{align*}
for every $\eta'>0$ with $\theta=\frac{1}{2(n-1)}$. Now it follows from Theorem \ref{th:ggn3} that
for almost every $x\in X$,
$$
\kappa_\Gamma(x,x_0)\le \frac{\dim(X)}{2\theta a(X)} =\frac{n}{n-1}\,,
$$
as claimed.
\end{proof}

\subsection{Simultaneous approximation on $3$-dimensional space}

We consider the standard action of $\SL_3(F)$ on the space $U=F^3\setminus \set{0}$ and the action
of $\SL_3(F)$ on the space $U\times U$ defined by $g(v,w)=(gv, \left(g^{-1}\right)^t w)$.
Given a lattice $\Gamma$ in $\SL_3(F)$,  we are interested in the problem of simultaneous 
Diophantine approximation in $U$, namely, 
we will seek to solve the inequalities 
$$
\norm{\gamma v-v^\prime}\le \epsilon,\;\; \norm{(\gamma^{-1})^t w-w^\prime} \le \epsilon
\quad\hbox{and}\quad \norm{\gamma} \le  \epsilon^{-\zeta}.
$$
We note that the action on $U\times U$ preserves the standard bilinear form $J(v,w)=\sum_{i=1}^3 v_iw_i$ and hence
each of the subvarieties 
$$
W_\alpha=\set{(v,w)\in U\times U(F)\,;\, J(v,w)=\alpha}.
$$

We establish the following  best possible result for Diophantine exponents,  uniformly for all lattices. 

\begin{Prop}\label{ex:3}
For an arbitrary lattice $\Gamma$ in $\SL_3(F)$, the exponent of Diophantine approximation for
$\Gamma$-action on $W_\alpha$ with $\alpha\ne 0$ is given by
$$
\kappa_\Gamma(x,x_0) = 5/2
$$
for almost every $x,x_0\in W_\alpha$. 
\end{Prop}

 \begin{proof}
We first observe that the action of $\SL_3(F)$ on $W_\alpha$ with $\alpha\ne 0$ is transitive.
Clearly $W_\alpha$ contains the orbit $\SL_3(F)\cdot (e_3,\alpha e_3)$ where $e_3=(0,0,1)^t$. 
This is in fact an equality, namely if $J(v,w)=\alpha$, then for some $g\in SL_3(F)$ we have $ge_3= v$ and $(g^{-1})^t \alpha e_3 =w$, or equivalently $g^t w=\alpha e_3$. Indeed the first condition amounts to the third column of $g$ being equal to $v$, and the second is solved by choosing the first two columns of $g$ to be a basis of $w^\perp$ (under the form $J$), linearly independent of $v$, and adjusting one of the basis vectors so that $g$ has determinant $1$.  

Since the action is transitive, $W_\alpha\simeq \SL_3(F)/H$ where $H$ 
is the copy of $SL_2(F)$ embedded in $\SL_3(F)$  in the upper left hand corner. 
It follows from Kazhdan's original argument \cite{k} (proving property $T$) that
the unitary representation of $H$ on $L^2_0(\Gamma \setminus \SL_{3}(F))$ is tempered.
Hence, by Corollary \ref{cor:ggn3},
$$
\kappa_\Gamma(x,x_0)=\frac{\dim(W_\alpha)}{a(W_\alpha)}=5/2
$$
for almost every $x,x_0\in W_\alpha$. Here we used that $a(W_\alpha)=2$ when $F=\R,\Q_p$
and $a(W_\alpha)=4$ when $F=\C$, which follows from the volume estimates (\ref{eq:sl}).
\end{proof}

\subsection{Inhomogeneous Diophantine approximation}

We turn to consider the action of a lattice subgroup $\Gamma$ of the affine group  $\SL_n(F)\rsemi F^n$ acting on $F^n$. 

%
\begin{Prop}
For $n \ge 3$  the exponent of Diophantine approximation for an arbitrary lattice
$\Gamma$ in $G=\SL_n(F)\rsemi F^n$ acting on $X=F^n\setminus \set{0}$ is estimated by 
\begin{equation}\label{eq:lll}
\frac{1}{n-1}\le \kappa_\Gamma(x,x_0) \le 1
\end{equation}
for almost every $x,x_0\in X$. 
\end{Prop}

We also remark that when  $n=2$ and the lattice satisfies a suitable spectral condition, the best
possible exponent  $\kappa_\Gamma(x,x_0)=\frac{1}{n-1}$ is achieved generically. 
This applies, for example, to the lattice $SL_2(\Z)\rsemi \Z^2$ acting on $\R^2$, and the lattice $SL_2(\Z[i])\rsemi \Z[i]^2$ acting on $\C^2$. For full details on this matter  we refer to \cite{GGN3}.

\begin{proof} 
We note that $X\simeq G/H$ where $H=\SL_n(F)$ is the stabiliser of the origin.
We introduce the norm on $G$ as in the proof of Proposition \ref{ex:4}. Then it follows from (\ref{eq:sl})
that 
\begin{equation}\label{eq:ax}
a(X)=\limsup_{t\to\infty} \frac{\log m_H(H_t)}{t} \asymp 
\left\{
\begin{tabular}{ll}
$n^2-n$ & \hbox{for $F=\R,\Q_p$,}\\
$2(n^2-n)$ & \hbox{for $F=\C$.}
\end{tabular}
\right.
\end{equation}
Now the lower bound in (\ref{eq:lll}) follows from (\ref{eq:llow}).
To establish an upper bound, we need to estimate decay of the operator norm of $\pi(\beta_t)$
where $\pi$ denotes the unitary representation of $H$ on $L^2_0(\Gamma\setminus G)$,
and $\beta_t$ is the Haar uniform average supported on the set $H_t$.
As in the proof of Proposition \ref{ex:4}, the representation $\pi^{\otimes (n-1)}$ is weakly contained
in the regular representation, and using
bi-invariance of $\beta_t$ under the maximal compact subgroup, the estimate of general matrix
coefficients by a power of the Harish-Chandra function from \cite{chh}, and H\"older's inequality,
we deduce that
\begin{align}\label{eq:b_gen}
\norm{\pi(\beta_t)} \le \frac{1}{m_H(H_t)} \int_{H_t}\Xi_{n}^{1/(n-1)}(h) dm_H(h)
\ll m_H(H_t)^{-\frac{1}{2(n-1)+\eta}}.
\end{align}
Hence, the upper bound in  (\ref{eq:lll}) now follows from Theorem \ref{th:ggn3}.
\end{proof}

\subsection{The variety of matrices with a fixed determinant}
Consider the variety of matrices with a determinant  $k\neq 0$:
$$
X=\set{x\in \hbox{M}_n(F)\,;\, \det (x) =k}\,.
$$
The group $G=\SL_n(F)\times \SL_n(F)$ acts transitively on $X$, via $(g_1,g_2)x=g_1xg_2^{-1}$. 
We introduce a norm on $G$ by embedding it diagonally in $\SL_{2n}(F)$.
Given a lattice $\Gamma$ in $G$, we are interested in estimating the exponents of Diophnatine
approximation for $\Gamma$-orbits in $X$, namely, to investigate existence of solutions
$(\gamma_1,\gamma_2)\in\Gamma$
of the inequalities
$$
\norm{\gamma_1 x\gamma_2^{-1}-x_0}\le \epsilon
\quad\hbox{and}\quad \norm{(\gamma_1,\gamma_2)} \le  \epsilon^{-\kappa}.
$$
In this setting we establish the following bounds:

\begin{Prop} 
For $n \ge 3$  the exponent of Diophantine approximation for an arbitrary irreducible lattice
$\Gamma$ in $G=\SL_n(F)\times \SL_n(F)$ acting on $X$ is estimated by 
$$
 \frac{n+1}{n}\le \kappa_\Gamma(x,x_0) \le \frac{n^2-1}{2n}
$$
for almost every $x,x_0\in X$. 

When $n=3$ the lower and upper bounds match, so that the optimal exponent is given by $4/3$. 
\end{Prop} 

We note that when $n=2$  it is also possible to obtain estimate of the Diophantine exponent, but these depend on the irreducible lattice chosen. We refer to \cite{GGN3} for these results.  

\begin{proof}
We observe that $X\simeq G/H$ where $H=\{(h,h);\, h\in \SL_n(F)\}\simeq \SL_n(F)$ is the stabiliser of the identity
matrix. The growth rate of $m_H(H_t)$ can be estimated as in the previous sections (cf. (\ref{eq:ax})),
and $\dim(X)=n^2-1$ when $F=\R,\Q_p$ and $\dim(X)=2(n^2-1)$ when $F=\C$. The lower bound
follows from (\ref{eq:llow}). To prove the upper bound, we need to estimate $\|\pi(\beta_t)\|$
where $\pi$ denotes the representation of $H$ on $L^2_0(\Gamma\setminus G)$.
The bound (\ref{eq:b_gen}) is valid for all unitary representations of $\SL_n(F)$ 
without invariant vectors. Hence, it applies in our case as well, and it follows from Theorem
\ref{th:ggn3} that with $\theta=\frac{1}{2(n-1)}$,
\begin{equation}\label{eq:lll}
\kappa_\Gamma(x,x_0)\le \frac{\dim(X)}{2\theta a(X)}=\frac{n^2-1}{n}
\end{equation}
for almost every $x\in X$.

However, in the present situation it is possible to give a better estimate. Indeed, since the lattice is
irreducible, the spectral decomposition of the representation of $G=\SL_n(F)\times \SL_n(F)$ in
$L^2_0(\Gamma \setminus G)$ involves tensor products $\pi_1\otimes \pi_2$ of irreducible infinite
dimensional representations of $\SL_n(F)$. It follows that the restriction of the associated $K$-finite
matrix coefficients to the diagonally embedded  group $H\simeq \SL_n(F)$ is in fact not just in
$L^{2(n-1)+\eta}(H)$ but in $L^{n-1+\eta}(H)$ for every $\eta>0$. 
For a detailed account of this argument we refer to \cite{GGN3}. 
It follows that the spectral estimate obtained is 
$$\norm{\pi(\beta_t)}\ll m_H(H_t)^{-\frac{1}{n-1+\eta}}.$$
Hence, (\ref{eq:lll}) holds with $\theta =\frac{1}{n-1}$, and this implies the upper bound in the proposition.
\end{proof}

\end{document}